\documentclass[reqno,12pt]{amsart}

\usepackage{amsmath,amsfonts,amssymb,amsthm, amscd}
\usepackage{mathrsfs,xcolor, latexsym, eucal, bbm}
\usepackage{enumerate,geometry}
\usepackage{graphicx}
\usepackage[all]{xy}

\newtheorem{thm}{Theorem}[section]

\newtheorem{prop}[thm]{Proposition}
\newtheorem{coro}[thm]{Corollary}
\theoremstyle{definition}

\theoremstyle{remark}

\numberwithin{equation}{section}
\DeclareMathOperator{\Hom}{Hom}

\def\+{{\dagger}}

\def\la{{\langle}}
\def\ra{{\rangle}}

\font\cyr wncyr10 at 14pt \def\Ch{\hbox{\cyr ch}}

\def\RR{{\mathbb R}}

\def\bx{\mathbf{x}}

\newcommand{\cK}{\mathcal{K}}

\newcommand{\cR}{\mathcal{R}}

\begin{document}

\title[Linear systems and CGM]{Linear systems over rings of measurable functions and conjugate gradient methods }
\date{\today}
\author[Lai]{King-Fai Lai}
\address[Lai]{School of Mathematical Sciences, Capital Normal University, Beijing 100048, China.}
\email{kinglaihonkon@gmail.com}
\begin{abstract} We study the conjugate gradient method for solving s system of linear equations with coefficients which are measurable functions and establish the rate of convergence of this method.
\end{abstract}
\maketitle

\section{Introduction}

The conjugate gradient method (CGM) is one of the most important 
iterative  methods used to solve a numerical linear system 
 $Ax=b$ ( \cite{HS 52}, \cite{Axe 94}). Couple with preconditioning it is often the most efficient method 
 (\cite{CX 07}, \cite{Ng 04}). 
 
 When the coefficient matrix $A$ of the system is not numerical but of the form
 $C+E$ where $C$ is a matrix with entries in complex numbers while $E$  is a matrix with entries which are random variables,   much work  have been on the statistical analysis of such systems. 
 
 We shall go in a different direction. 
 The goal of this paper is to study the algebraic aspects of the  computation.
 We want to apply CGM \textit{directly} to a linear system $Ax=b$ in which $A,x,b$ have entries which are  real valued measurable functions.
 From a computational point of view we are in a totally new direction. 
 We are proposing to calculate the functions as elements of a ring and try to obtain solutions to very large systems as functions and not
  to evaluate the system at  a few selected points, 
  compute numerically a solution at a these selected points  and pretend that these few numerical values give in fact the whole function which is the solution of the given system.  
 As it is usual to work with  measurable functions equivalent up to
 sets of measure zero and we need to invert strictly positive  elements 
 in order that CGM works, we replace 
  the ring of   measurable functions  with a commutative real algebra 
 $\cR$ constructed from it by taking quotient (\cite{Ste 10})
 and localization (\cite{Bru 79}). 
We give an abstract characterization of $\cR$ and  
 call it a Riesz algebra (to compare it with a similar structure \cite{Fre 74}). In order to 
 establish the rate of convergence of CGM by Krylov's method in this case we shall see that we need all the rich structures of a Riesz algebra to get results on positive definite quadratic forms and min-max estimates which are standard over fields. This will show that the Riesz algebra is the right place for computational linear algebra for functions.
 
 We thank Professor Wen-Fong Ke for stimulating conversations on this paper and 
 the National Center of Theoretical Science (South) for the support of a short visit to National Cheng Kung University Taiwan during which this work is done. 
 
 \section{Riesz algebra}
 
We give the definition of a Riesz algebra. 
 
By a \textit{partially ordered ring} we mean 
 a ring with identity  equipped with a partial order $\leq$ such that
(1) for $x,y \in R$, if $ x \leq y$  then
$x + a\leq y+a$ for any $ a \in R$,
(2) if $x \geq 0$ and $y \geq 0$ then $xy \geq 0$ and 
(3) $a^2 \geq 0$ for all $ a\in R$.
Write $a \leq b$ if $b-a \geq 0$. 

Let $\RR$ be the field of real numbers. 
A partially ordered $\RR$-algebra $R$ is a partially ordered ring such that
(1) if $a \in R$ and $\alpha$ is a non-negative real number then $\alpha a \geq 0$ 
(i.e. $R$ is a partially ordered vector space); and  
(2) the order of $R$ extends that of the real numbers $\RR$. i.e. if $\alpha \geq 0$ is a real number and $1$ is the identity in $R$ then
$\alpha 1 \geq 0$ in $R$. We shall write $\alpha $ for $\alpha 1$. 

Say an element $a$ in
 a partially ordered ring $R$ is positive and write $a > 0$ if $a \geq 0$ and $a \ne 0$. 
We say $a$ is \textit{strictly  positive} and write $a \succ 0$ if
if $a > 0$ and $a$ is invertible in $R$.
We say that a partially ordered $\RR$-algebra $R$ is 
\textit{strictly archimedean} if for any $a,b \in R$, the condition
$r b \prec a$ holds for any $r \in \RR$ implies that $b =0$.

A \textit{lattice} is a partially ordered set $(A, \leq)$ such
that the supremum $\sup \{a,b\}$ and infinmum $\inf \{a,b\}$ exist for any pair of  elements $a,  b \in A$. 
We write 
$|a|$ for $\sup \{a, -a\}$.
A partially ordered ring $(R, +, \cdot)$
which is also a lattice is called a lattice-ordered ring (\cite{Ste 10} \S 3.1). 
A partially ordered vector space which is also a lattice is called a Riesz space 
(\cite{Fre 74} \S14A).

We shall call a partially ordered   strictly
archimedean   $\RR$-algebra $R$ which is also a lattice  a \textit{Riesz algebra}. 
If for every
  $a \succ 0$ in $R$  there exists in $R$
an element $b \succ 0$  such that $b^2 =a$, we say that $R$ is a \textit{real Riesz algebra}. 
Write $b$ as $\sqrt{a}$
(cf. \cite{Jac 85} I p.308). When $R$ is also commutative ring 
we call it a commutative real Riesz algebra. 

We say  a $n \times n$ symmetric  matrix $A$ with entries in 
a commutative real Riesz algebra
$R$
is  positive definite if
 for any non zero  vector $y$ in the $R$-module $R^n$ 
 of column $n$-vectors we have $y^T A y \succ 0$ in $R$.
 For $ x,y, \in R^n$ we write
$\la x, y \ra_A$ for $ x^T A y$. 
Say two vectors $x, y$ are \textit{conjugate} or $A$-conjugate or
$A$ perpendicular if $\la x,y \ra_A = 0$.
If $ x \ne 0$ put $ \|x\|_A = \sqrt{\la x, x \ra_A}$ and 
set  $ \|0\|_A =0$.

\begin{prop}(Schwartz inequality) Let $A$ be a positive definite 
symmetric matrix with entries in 
a commutative real Riesz algebra $R$. 
For any non zero vectors $x, y \in R^n$  then we have in $R$
$$|\la x, y \ra_A| \prec \|x\|_A \|y\|_A.$$
\end{prop}
\begin{proof}As $\|x\|_A$ is in $R$ we have
$$ \la \|x\|_A y- \|y\|_A x, \|x\|_A y- \|y\|_A x \ra_A
= 2 \|x\|_A^2 \|y\|_A^2 - 2 \|x\|_A \|y\|_A \la x, y \ra_A. $$
As $A$ is positive definite we get
$2\|x\|_A \|y\|_A \la x, y \ra_A \prec  2 \|x\|_A^2 \|y\|_A^2$.
As strictly positive elements are invertible ina Riesz algebra,
 $2, \|x\|_A, \|y\|_A$ are invertible and it follows that 
$$\la x, y \ra_A \prec \|x\|_A \|y\|_A.$$
A similar calculation of
$ \la \|x\|_A y+ \|y\|_A x, \|x\|_A y+ \|y\|_A x \ra_A $
shows that $-\la x, y \ra_A \prec \|x\|_A \|y\|_A$.
\end{proof}
\begin{coro}(Triangle inequality) For non zero vectors $x, y \in R^n$ we have
 $$\|x + y\|_A \prec  \|x\|_A + \|y\|_A.$$
\end{coro}

We continue to write $R$ for a commutative real Riesz algebra and 
$R^\times$ for the subgroup of invertible elements in $R$. We have just seen that 
$\la x, y \ra_A$ is a symmetric bilinear form on the $R$-module $R^n$.
In general we can consider a symmetric bilinear form 
$b: M \times M \to R$ on a
 finitely generated  $R$-module $M$.
 For a submodule  $S$ of $M$ we write 
 $b \mid_S$ for the restriction of $b$ to $S$ and 
 $$S^\perp = \{x \in M: b(x,y)=0 \; \forall y \in S \}.$$
 We say  a symmetric bilinear form 
$b$ on a
 finitely generated  $R$-module $M$ is non-degenerate if
 \begin{enumerate} \item  $b(x,y) =0$ for all $y \in M$ $\Rightarrow x =0$,
 \item  if $f \in \Hom_R(M, R)$ then there exists $x_f \in M$ such that $f(y)= b(x_f, y)$ for any $y \in M$.
 \end{enumerate}
 
Just as in the case over fields it can be proved that 
a symmetric bilinear form on a finite rank
free $R$-module is non-degenerate
if and only if its matrix associated to any basis is invertible.
Moreover the following results are standard.

 \begin{prop} \label{lem:basis} 
Let $b$ be a 
symmetric bilinear form on a finitely generated $R$-module $M$. 
Then
\begin{enumerate}
\item
 $M= S \perp S^\perp$ meaning
 $M= S \oplus S^\perp$ and
 $b = b \mid_S \oplus b \mid_{S^\perp}$, that is
 for $x,y \in S$ and $u,v \in S^\perp$ we have
 $$b(x+u, y+v) = (b \mid_S)(x,y) + (b \mid_{S^\perp})(u,v).$$
 \item  Put
$N=\{x \in M: b(x,x) \not\in R^\times \}$. Then $N$ is a an $R$-submodule of $M$.
\item  If $N \ne M$ then there exist $x_1, \cdots, x_k
\in M$ such that  $b(x_i, x_i) \in R^\times$ and
$$M = Rx_1 \perp  \cdots \perp  Rx_k \perp  N.$$
\end{enumerate}
\end{prop} (\cite{Jac 85} I Theorem 6.1, \cite{Bae 78}, \cite{Knu 91}.)
 
 \section{Algebra of measurable functions}
 
 We fix a measure space $(X, \Sigma, \mu)$; here $\Sigma$ is a
$\sigma$-algebra of subsets of $X$ and we assume that   $\mu(X)$ is  finite.
We write a.e. for almost everywhere. 

The set $M$ of all real valued measurable functions on $X$ 
is a  commutative $\RR$-algebra.
The set $N$  consisting of functions which are zero a.e. is
an ideal in $M$.
Let $R$ denote the quotient ring $M/N$.
Write  $\la f \ra$ for the image of $f \in M$ in $R$. 

Set
 $\la f \ra \geq 0$ if $\mu\{f <0\}=0$.
Then $R$ is a partially ordered ring 
and a Riesz space (see \cite{Fre 74} \S62F(c) 
\S 62G;  \cite{Ste 10} \S 3.1).
To say that $\la f \ra \ne 0$ is saying $f \not\in N$, 
i.e. $\mu\{f \ne 0\} \ne 0$.
We write  $\la f \ra > 0$  
to mean $\la f \ra \geq 0$ and $\la f \ra \ne 0$.

We shall  write $\la f \ra \succ 0$ if
$\mu\{f \leq 0\}=0$.
Let $S$ be the set of all $\la f \ra$ in $R$ such that either 
$\la f \ra \succ 0$ or $\la -f \ra \succ 0$. 
Then
$S$ is a multiplicative set in $R$.
We localize $S$ to get a  ring of quotients $\cR$
 in which every element in $S$ is invertible. We can represent an element of $\cR$ as $\frac{\la g \ra}{\la f \ra}$ with 
$\la g \ra$ in $R$ and $\la f \ra$ in $S$. We say 
$\frac{\la g \ra}{\la f \ra} \geq 0$ if $\la g \ra \geq 0$.
This defines a partial order making
 $\cR$ a partially ordered ring. 
 For $a = \frac{\la g \ra}{\la f \ra}$ we shall write 
 $a \succ 0$ if $\la f \ra \la g \ra \succ 0$. Then 
  $a \succ 0$ if and only if $a > 0$ and $a$ is invertible in 
  $\cR$. 

Suppose $a = \frac{\la g \ra}{\la f \ra}$ and $\la f \ra \succ 0$. For $x \in X$ we set 
$h(x)=0$ if $f(x) \leq 0$ and equals to $ \frac{g(x)}{f(x)}$ 
otherwise. Then $h$ is measurable
(\cite{HS 65} \S 11 Theorem 11.8) and $h = \frac{g}{f}$ a.e.
We set $\sup \{a,0\}$ to be the image of  
$\sup \{h,0\}$ in $\cR$. Similar definition is given when 
$\la -f \ra \succ 0$.
Clearly if $a \in R$ then this agrees with the definition of $\sup$ in $R$ (as in \cite{Fre 74} 14G (c)). 
This defines the structure of a lattice in
 $\cR$. 

 We summarize our discussion in the following proposition. 
 
 \begin{prop} $\cR$ is a real Riesz algebra. \end{prop}

 \textit{From now on $\cR$ will always denote this Riesz algebra}. 
 We also say $\cR$ is the Riesz algebra on the measure space $X$.


Let $A$ be a $n \times n$ matrix with entries in the Riesz algebra of 
measurable functions on a measure space $X$. Then it is known that 
 its eigenfunctions can be ordered (\cite{KLW 13}, \cite{AGZ 09}) 
$$0 \leq \lambda_1 (x) \leq \cdots 
\leq \lambda_n (x), \;\; x \in X.$$

\begin{prop} Let $A$ be a $n \times n$ positive definite 
symmetric  matrix with entries in the Riesz algebra $\cR$ 
on the measure space $X$. Let $\mathbf{y}_j$ (in $\cR^n$) be the eigenvector
of $A$ with eigenvalue $\lambda_j$. 
Then 
\{$\mathbf{y}_j\}$  form a basis of $\cR^n$. 
\end{prop}
\begin{proof}
If $A$ is positive definite then in the notation of 
the proposition \ref{lem:basis} the space $N$ associated to the bilinear form $\la \cdot, \cdot \ra$ is zero. 
And so
$$\cR^n = \cR \mathbf{y}_1 \perp  \cdots \perp  \cR \mathbf{y}_n.$$
\end{proof}

 Let $A$ be a $n \times n$ positive definite 
symmetric   matrix with entries in the Riesz algebra $\cR$ on the measure space $X$. Write
$\lambda_{\min}$ for its minimal
 eigenfunction 
and $\lambda_{\max}$ for its maximal eigenfunction. Put
$$\underline{\lambda} = \inf_{x \in X} \lambda_{\min}(x)
\qquad  \overline{\lambda}= \sup_{x \in X} \lambda_{\max}(x).$$
From a polynomial $q= \sum_i \alpha_i(x) T^i$ in $\cR[T]$ we get a 
function $q(x, t)$ on $X \times [\underline{\lambda}, \overline{\lambda}]$. 
Let us write 
$$M^A ( q)=
 \sup_{\substack{\underline{\lambda} \le t \le 
 \overline{\lambda}\\ x \in X}}|q(x, t)|.$$
 \begin{prop}\label{pro:mq} Notations as above.
If 
$q$ is a polynomial over $\cR$, and $\mathbf{x} \in \cR^n$ then 
in $\cR$ we have
$$\Vert q(A) \mathbf{x}\Vert_A \leq 
M^A(q)
\cdot \Vert \mathbf{x} \Vert_A$$
\end{prop}
\begin{proof}
Take $\mathbf{y}_j$ as in the previous proposition, write $\mathbf{x}=\sum_j \beta_j \mathbf{y}_j$ then 
$$q(A)\mathbf{x} = \sum_i \beta_i q(\lambda_i)\mathbf{y}_i,$$
and 
\begin{align*}
\Vert q(A) \mathbf{x}\Vert_A^2 
&= (\sum_i \beta_i q(\lambda_i)\mathbf{y}_i)^T A 
(\sum_j \beta_j q(\lambda_j)\mathbf{y}_j)\\
&= \sum_{i,j}  q(\lambda_i) q(\lambda_j)
\beta_i \beta_j \mathbf{y}_i^T A \mathbf{y}_j \\
&\leq (M^A(q))^2 \sum_{i,j} \beta_i \beta_j \mathbf{y}_i^T A \mathbf{y}_j \\
&=  (M^A(q))^2 \Vert \mathbf{x} \Vert_A^2
\end{align*}
\end{proof}

\section{Conjugate gradient method}

Let $A$ be a $n \times n$ positive definite 
symmetric   matrix with entries in the Riesz algebra $\cR$ 
on a measure space $X$ and
 $\mathbf{b}$  a  vector in $\cR^n$. 
 We try to find iteratively  a solution
 $\mathbf{x} \in \cR^n$ of the linear system $A\mathbf{x} =\mathbf{b}$. 
 
We  start with any point $\bx_0$ in $\cR^n$ and take
$\mathbf{r}_0= \mathbf{p}_0 = \mathbf{b} - A \bx_0$.

At the $k$-th step we compute 
\begin{align*}
\bx_{k} &= \bx_{k-1} + \alpha_{k-1} \mathbf{p}_{k-1}, \\
\mathbf{r}_k &= \mathbf{b} - A \bx_k, \\
\mathbf{p}_k &= \mathbf{r}_k - 
\frac{\mathbf{r}_k^T A \mathbf{p}_{k-1}}{\mathbf{p}_{k-1}^T A \mathbf{p}_{k-1}} \mathbf{p}_{k-1},\\
\alpha_k &= 
\frac{\mathbf{r}_k^T \mathbf{p}_k}{\mathbf{p}_k^T A \mathbf{p}_k}.
\end{align*}
\noindent Note that  $\mathbf{p}_k^T A \mathbf{p}_k$ is strictly 
positive and by construction is invertible in the commutative algebra 
$\cR$. 
The  term 
$\mathbf{r}_k = \mathbf{b} - A \bx_k$
is called the residue term. If $\mathbf{r}_k = 0$ 
then $\bx_k$ is the solution we sort after and the computation stops. 
In this case we shall say that CGM is \textit{successful}
 for the system 
$A\mathbf{x} =\mathbf{b}$. 
CGM is only called for when $n$ is huge compare with the size of the computing facility. In this case the program will often be stopped when 
$|\sup_{ X} \mathbf{r}_k^T \mathbf{r}_k|$ is smaller than a 
pre-determined small number. 

The term $\alpha_k$ is called the control term. 
If $\alpha_k$ is in the 
set $S$ of positive definite or negative definite elements
we say it is acceptable and we continue. 
We shall also say  in this case that the CGM is \textit{feasible} for the given system at the $k$-th step. 
If $\alpha_k$ is not in $S$ we stop the 
program. The point is this. If $\alpha_k(x) =0$ at $x \in X$
we can say $(\mathbf{b} - A \mathbf{x}_k)(x)=0$ and we find a solution
$\bx_k(x)$ of the 
system $A(x) \bx(b)= \mathbf{b}(x)$ at the point $x$. But this does tell us if the function  $\bx$ gives the solution at
other points in the space $X$. The problem being that the algebra 
$\cR$ has plenty of non-zero zero divisors. This shows the difficulty of solving for functions. But if we do find a function solution we have a global solution rather than a solution at a point in the space $X$. 
This shows the convenience of working in an abstract Riesz algebra. 

This computation will have a failure set
$$\Upsilon_{k} = \{x \in X : \mathbf{p}_{k-1}^T A \mathbf{p}_{k-1}(x)=0\}$$
which is of measure zero by the choice of $A$.
As countable union of sets of measure zero has 
 measure zero, we know that
 we can continue outside of $\cup_k \Upsilon_k $, that is the 
 computation of the conjugate gradient can be done a.e.
 
By a \textit{Krylov module} of the $n \times n$
  matrix $A$ we mean a $\cR$-module
$\cK (A, \mathbf{y}, k) $ spanned over $\cR$ 
by the set $\{\mathbf{y}, A\mathbf{y}, \cdots, A^k \mathbf{y}\}$ 
where $\mathbf{y}$ is a  vector  
in $\cR^n$ and $k$ is an integer. 

\begin{thm}\label{pro:cgp} In the above notations  assuming 
that CGM is feasible 
 for  the linear system $A \bx = \mathbf{b}$. Then
\begin{enumerate}
\item $\mathbf{p}^T_i \mathbf{r}_j =0$ for $0 \leq i < j \leq k$.
\item $\mathbf{r}^T_i \mathbf{r}_j =0$ for $ i \ne j$, $0 \leq i,  j \leq k$.
\item $\la \mathbf{p}_i, \mathbf{p}_j \ra_A =0$  for $ i \ne j$, $0 \leq i,  j \leq k$.
\item $\cK (A, \mathbf{r}_0, k) $ is spanned by $\{\mathbf{r}_0, \cdots, \mathbf{r}_k\}$ or by
$\{\mathbf{p}_0, \cdots, \mathbf{p}_k\}$.
\end{enumerate}
\end{thm}
\begin{proof}
We prove   the induction step from $k$ to $k+1$.

For part $(1)$ - from the definition of $\mathbf{r}_k$ we get 
$$\mathbf{r}_{k+1} = \mathbf{r}_{k} - \alpha_{k} A \mathbf{p}_{k}.$$
From this and the definition of $\alpha_{k}$ it follows that 
$\mathbf{p}^T_k \mathbf{r}_{k+1} =0$ 
and also 
by the induction hypothesis for parts $(1)$ and $(3)$,
$\mathbf{p}^T_i \mathbf{r}_{k+1} =0$ for $0 \leq i  \leq k-1$. 

For part $(2)$ we get from part $(1)$ that  
$\mathbf{p}^T_i \mathbf{r}_{k+1} =0$ for $0 \leq i \leq k$. By 
the induction hypothesis of part $(4)$ which says 
that $\{\mathbf{r}_0, \cdots, \mathbf{r}_k\}$ and
$\{\mathbf{p}_0, \cdots, \mathbf{p}_k\}$ span the same module, 
we conclude that
$\mathbf{r}^T_i \mathbf{r}_{k+1} =0$ for $0 \leq i  \leq k$. 

For part $(3)$ - we write 
$\beta_k =(\mathbf{r}_k^T A \mathbf{p}_{k-1})(\mathbf{p}_{k-1}^T A \mathbf{p}_{k-1})^{-1}$,  the construction says
$\mathbf{p}_{k+1} = \mathbf{r}_{k+1} - \beta_{k+1} \mathbf{p}_{k}$. 
And so
$$\mathbf{p}_{k+1}^T A \mathbf{p}_{i}
= \mathbf{r}_{k+1}^T A \mathbf{p}_{i} -
\beta_{k+1} \mathbf{p}_{k}^T A \mathbf{p}_{i}.
$$
For $0 \leq i < k$ assuming that the program can continue and the 
$\alpha_i$ ($0 \leq i \leq k-1$) are acceptable, 
 from part $(2)$ proved above and by 
the induction hypothesis of part $(3)$ it follows that 
$\mathbf{p}_{k+1}^T A \mathbf{p}_{i}
=0$. By construction $\mathbf{p}_{k+1}^T A \mathbf{p}_{k}
= 0$. 

For part $(4)$ we start with the induction hypothesis that 
either $\{\mathbf{r}_0, \cdots, \mathbf{r}_k\}$ or 
$\{\mathbf{p}_0, \cdots, \mathbf{p}_k\}$ spans  
$\cK (A, \mathbf{r}_0, k) $. Then the formulas
$\mathbf{r}_{k+1} = \mathbf{r}_{k} - \alpha_{k} A \mathbf{p}_{k}$, and 
$\mathbf{p}_k = \mathbf{r}_k - \beta_k \mathbf{p}_{k-1}$,
tell us that $\mathbf{r}_{k+1}, \mathbf{p}_{k+1}$
are in $\cK (A, \mathbf{p}_0, k+1) $. 

For $A^k \mathbf{r}_0$ in  $\cK (A, \mathbf{r}_0, k) $ we can write 
$A^k \mathbf{r}_0 = \sum_{i=0}^k \gamma_i \mathbf{p}_{i}$ with 
$ \gamma_i \in \cR$. Then 
$A^{k+1} \mathbf{r}_0 = \sum_{i=0}^k \gamma_i A \mathbf{p}_{i}$. 
From $\mathbf{p}_{i}$ in 
$\cK (A, \mathbf{r}_0, k) $ we get $A \mathbf{p}_{i}$ and so 
$A^{k+1} \mathbf{r}_0$ is in $\cK (A, \mathbf{r}_0, k+1) $. 

Moreover parts $(2)$  tells us that the vectors
$\mathbf{r}_0, \cdots, \mathbf{r}_{k+1}$ are linearly independent so are 
$\mathbf{p}_0, \cdots, \mathbf{p}_{k+1}$. Thus
$\cK (A, \mathbf{p}_0, k+1) $ is free of rank $k+1$ over $\cR$.
\end{proof}

\noindent Remark. The theorem tells us that the CGM stops before the 
$n+1$ step either when it is successful or when it is not feasible.

\begin{prop}\label{pro:cge}In the above notations  assuming 
that CGM is feasible 
 for  the linear system $A \bx = \mathbf{b}$ 
 up to the $\ell$-th step yielding an output ${\bx}_\ell$. 
Then for $k < \ell$ we have 
$$\| \bx_{\ell}-\bx_k\|_A
\le   \inf \{\| \bx_{\ell}-\bx\|_A: 
\bx\in \bx_0+\cK(A,\mathbf{r}_0,k-1)\}.$$
\end{prop}
\begin{proof}
From the construction we have
$\bx_{k} = \bx_{k-1} + \alpha_{k-1} \mathbf{p}_{k-1}$. It follows that
(i) $\bx_{k} = \bx_{0} + \alpha_{0} \mathbf{p}_{0} + \cdots +\alpha_{k-1} \mathbf{p}_{k-1}$, and 
(ii) $\bx_{\ell} = \bx_{k} + \alpha_{k} \mathbf{p}_{k} + \cdots +
\alpha_{\ell-1} \mathbf{p}_{\ell-1}$.

From (i) we see that $\bx_{k} \in \bx_0+\cK(A,\mathbf{r}_0,k-1)$
and so 
for any $\bx\in \bx_0+\cK(A,\mathbf{r}_0,k-1)$
 we get $\bx_k-\bx$ is in $\cK(A,\mathbf{r}_0,k-1)$ which is spanned 
 by $\mathbf{p}_{0}, \cdots,  \mathbf{p}_{k-1}$. While (ii) says 
 $\bx_{\ell} - \bx_{k}$ is in the submodule spanned by 
$\mathbf{p}_{k}, \ldots, \mathbf{p}_{\ell-1}$.  By
 theorem \ref{pro:cgp}(3) that, $\langle \bx_{\ell}-\bx_k,\bx_k-\bx\rangle_A=0$ and so the proposition follows from 
 $$ \| \bx_{\ell}-\bx\|^2_A=\| \bx_{\ell}-\bx_k\|^2_A+
\| \bx_k-\bx\|^2_A\ge\| \bx_{\ell}-\bx_k \|^2_A.$$
\end{proof}

\section{Rate of convergence}

We continue to write  $\cR$ for
 the Riesz algebra  on the measure space $X$. 
We are interested in the set $\cR[T]_k^1$  of polynomials 
in the variable $T$  over $\cR$ of degree $\leq k$ with constant term
$1$. 
Let us consider an element 
$p= 1 + a_1 (x) T + \cdots +  a_k(x) T^k$ in $\cR[T]_k^1$ as a function  on $X \times [a,b]$
for some $0 < a < b$. Then we can consider  the real valued map 
$M$ on  
$\cR[T]_k^1$
given by
$$M ( p(x, T))=
 \sup_{\substack{a\le t \le b\\ x \in X}}|p(x, t)|.$$

We can apply the standard result in approximation theory at least
pointwise in $X$ to find a lower bound for $M$.
Namely let
 $$C_k(T)=
\frac{1}{2}[(T+\sqrt{T^2-1})^k+(T+\sqrt{T^2-1})^{-k}]$$
denotes the Chebyshev polynomial of degree $k$. And put
$$\Ch(x)= C_k (\frac{b+a-2T}{b-a}) / C_k (\frac{b+a}{b-a}).$$
This is a polynomial in $T$ with real coefficients.
Let $\RR[X]_k^1$ denote the set   of polynomials over $\RR$ of degree $\leq k$ with constant term
$1$. For a real polynomial $p$ write
$$m(p) = \sup_{a\le t \le b} |p(t)|$$
Then by approximation theory (\cite{Riv 90}; \cite{Axe 94} Appendix B) we have
$$\sup_{a\le t \le b} \Ch(t) = \inf_{p \in \RR[X]_k^1} m(p). $$
But
$$\sup_{a\le t \le b}|\Ch(t)|
\leq 2\left(\frac{\sqrt{\kappa}-1}{\sqrt{\kappa}+1}\right)^k,
\qquad \text{where}\quad
\kappa = \frac{b}{a}.$$
So
$$\inf_{p \in \RR[T]_k^1} m(p)
\leq 2\left(\frac{\sqrt{\kappa}-1}{\sqrt{\kappa}+1}\right)^k. $$

Now $\cR[T]_k^1 \supseteq \RR[T]_k^1$. And for
$p \in \RR[T]_k^1$ we have $M(p) = m(p)$. Thus
$$
 \inf\limits_{p \in \cR[T]_k^1} M(p)
 \leq \inf\limits_{p \in \RR[T]_k^1} M(p)
 = \inf\limits_{p \in \RR[T]_k^1} m(p)
 \leq 2\left(\frac{\sqrt{\kappa}-1}{\sqrt{\kappa}+1}\right)^k,
\qquad \text{where}\quad
\kappa = \frac{b}{a}.
$$

\begin{thm}\label{thm:rate}
Let    $A$ be a $n \times n$ matrix 
with entries in the Riesz algebra of 
 a measure space $X$. 
Write
$\lambda_{\min}$ for its minimal
 eigenfunction and 
 $\lambda_{\max}$ for its maximal eigenfunction. Put
$$\underline{\lambda} = \inf_{x \in X} \lambda_{\min}(x)
\qquad  \overline{\lambda}= \sup_{x \in X} \lambda_{\max}(x),$$
and
$\kappa = \overline{\lambda}/\underline{\lambda} $. 
Assume that the conjugate gradient method 
 for  the linear system $A \bx = \mathbf{b}$ 
is successful and yields an exact solution $\bx^*$. 
Then  the $k$-th
feasible output ${\bx}_k$ satisfies the following estimate
$$\| \bx^* - \bx_k \|_A
\leq 2 \left( \frac{\sqrt{\kappa} - 1}{\sqrt{\kappa} + 1}\right)^{k} \| \bx^* - \bx_0\|_A. $$
\end{thm}
\begin{proof} 
From theorem \ref{pro:cgp}(4)we know
 that for any $\bx\in \bx_0+\cK(A,\mathbf{r}_0,k-1)$
there exists a polynomial $p_k(T)$ in $\cR[T]$ of degree $k-1$ such that
$ \bx=\bx_0+p_k(A)\mathbf{r}_0$.
Recall that $\mathbf{r}_0=  \mathbf{b} - A \bx_0$.  
Hence
$$
\bx^*-\bx = A^{-1}\mathbf{b}-\bx_0-p_k(A)\mathbf{r}_0 
 = q_k(A)A^{-1}\mathbf{r}_0=q_k(A)(\bx^*-\bx_0),
$$
where $q_k( T)=1-T p_k(T)$ is a polynomial of degree $k$ and $q_k(0)=1$.

Applying proposition \ref{pro:cge}
we get
$$\| \bx^*-\bx_k\|_A
\le   \inf \{\| \bx^*-\bx\|_A: 
\bx\in \bx_0+\cK(A,\mathbf{r}_0,k-1)\}
\le \inf\limits_{q\in \cR[T]_k^1}\| q(A)(\bx^*-\bx_0)\|_A.$$
Using  \ref{pro:mq} we have 
$$\| \bx^*-\bx_k\|_A \le 
\inf\limits_{q\in \cR[T]_k^1}
M^A(q) \| \bx^*-\bx_0\|_A
\le   2\left(\frac{\sqrt{\kappa}-1}{\sqrt{\kappa}+1}\right)^{k}
\| \bx^*-\bx_0\|_A.$$
\end{proof}
\noindent This is the same estimate as in the numerical case as given in  \cite{Axe 94} \S 13.2.1.

\section{Conclusions}

We have seen to what extend CGM can be used to solve a 
large linear system over 
the algebra of measurable functions on a measure space. 
The aim is to try to find a function which is a solution of the system rather than just doing a point-wise computation and getting only the values of the solution function at a few selected points.  
In the process we see that we need the theory of quadratic forms over rings and  an order structure on the ring of measurable functions for estimates. The result is a Riesz algebra. 
It is clear that much can be done about computational linear algebra over  a Riesz algebra - for example we can develop preconditioning methods for Wiener-Hopf integral equations in this context.

\end{document}